\documentclass[10pt,oneside]{amsart}
\usepackage[a4paper, width=5in, height=6.9in]{geometry}
\usepackage{amsmath,amssymb,amsthm,booktabs,longtable,enumerate}
\usepackage{times}
\usepackage[loadonly]{enumitem}
\usepackage{comment}
\usepackage{graphicx}
\usepackage[all]{xy}
\usepackage{caption}
\usepackage{color}
\usepackage[pdftex]{hyperref}
\theoremstyle{plain}
\makeatletter 

\@addtoreset{table}{section}
\makeatother  
\newtheorem{theorem}{Theorem}[section]
\newtheorem*{theorem*}{Theorem}
\newtheorem*{remark*}{Remark}
\newtheorem{lemma}[theorem]{Lemma}
\newtheorem{corollary}[theorem]{Corollary}
\newtheorem{proposition}[theorem]{Proposition}
\newtheorem{definition}[theorem]{Definition}

\newtheorem{example}[theorem]{Example}

\newtheorem{remark}[theorem]{Remark}
\newtheorem{observation}[theorem]{Observation}
\newtheorem{question}[theorem]{Question}

\newcommand{\Isom}{\mathop{\mathrm{Isom}}\nolimits}

\newcommand{\diag}{\mathop{\mathrm{diag}}\nolimits}

\newcommand{\R}{\mathbb{R}}

\newcommand{\Z}{\mathbb{Z}}

\makeatletter
    
    \@addtoreset{equation}{section}
  \makeatother
\makeatletter
\renewcommand{\subsection}{\@startsection{subsection}{2}{0mm}%
  {\baselineskip} 
  {0.5\baselineskip} 
  {\bfseries\itshape\large}} 
\makeatother
\usepackage{enumerate}%
\makeatother
\begin{document}

\title[Proper actions with bornology and coarse geometry]{A characterization of proper actions with bornology and coarse geometry}
\author{Hiroaki Nagaya}
\subjclass[2020]{
Primary 57S30, 
Secondary 
46A08, 
46A17, 
53C23, 
51F30} 
\keywords{proper action; properly discontinuous; bornological space; Coarse space}

\address[H.~Nagaya]{%
	Graduate School of Advanced Science and Engineering, Hiroshima University, 
    1-3-1 Kagamiyama, Higashi-Hiroshima City, Hiroshima, 739-8526, Japan.
        }
\email{nagayahiroaki@hiroshima-u.ac.jp}

\begin{abstract}
    In 1961, Palais showed that every smooth proper Lie group action on a smooth manifold admits a compatible Riemannian metric on the manifold such that the action becomes isometric.
     In 2006, Yoshino studied a continuous proper action of a locally compact Hausdorff group on a locally compact Hausdorff space, and showed that the space carries a compatible uniform structure making the action equi continuous in an appropriate setting.
     In this paper, we focus on bornological proper actions on bornological spaces and prove that the space admits a compatible coarse structure such that the action becomes equi controlled. 
\end{abstract}

\newgeometry{width=5in, height=7.7in}
\maketitle

\tableofcontents

\section{Introduction and Main Results}\label{section:introduction}

Let $L$ be a locally compact Hausdorff group, $X$ a locally compact Hausdorff space and $\rho$  a continuous $L$-action on $X$.
Throughout this paper, the continuous $L$-action $\rho$ on $X$ is called \emph{proper} if
the subset 
\[
L_{C, C'} :=\{l\in L\mid lC\cap C'\neq \emptyset \}
\]
of $L$ is compact for any pair $(C,C')$ of compact subsets of $X$.
In addition, the proper $L$-action $\rho$ on $X$ is said to be \emph{properly discontinuous} if the group $L$ is discrete. 

The theories of properness and properly discontinuous actions play a fundamental role in the study of $(G,X)$-manifolds of group actions, particularly in the theory of Clifford–Klein forms (see \cite{Kobayashi-unlimit} for details).

The studies of proper actions and properly discontinuous actions behave quite differently depending on whether a metric is fixed or not.
Firstly, we focus on the case where a metric is fixed.
In this setting, let us consider the following question.
\begin{question}\label{question:metricvercharacter}
  Fix a metric $d$ on $X$ such that the $L$-action $\rho$ on $X$ is compatible with $d$ (i.e.~ isometric).
  How can the properness of $\rho$ be characterized in terms of the behavior of $L$-orbits?
\end{question}

In the setting where $(X,d)$ is a Heine–Borel metric space (i.e.~ every bounded closed subset is compact), the following well-known proposition gives an answer to Question \ref{question:metricvercharacter}.

\begin{proposition}\label{prop:isometric}
    Assume that  the space $X$ is equipped with a Heine–Borel metric $d$ and the $L$-action $\rho$ on $(X,d)$ is isometric.
    Then the following conditions on $\rho$ are equivalent.
    \begin{enumerate}
        \item\label{prop:isometric:proper} The $L$-action $\rho$ on $X$ is proper.
        \item\label{prop:isometric:closed} For any point $x\in X$, the orbit $L\cdot x$ is closed in $X$ and the isotropy $L_x\subset L$ is compact.
    \end{enumerate}
\end{proposition}

Moreover, the following is also well-known.

\begin{proposition}[cf.~ {\cite[Theorem 2.2]{Kramer2022}}, {\cite[Theorem 5.3.5]{Ratcliffe2019}}]\label{prop:isomdisconti}
    Assume that the space $X$ is equipped with a Haine-Borel metric $d$.
    Consider the natural isometric transformation group $\Isom (X)$ with compact open topology.
    Then the natural $\Isom(X)$-action on $X$ is proper.
    In particular, the action of any discrete subgroup $\Gamma$ of $\Isom (X)$ is properly discontinuous.
\end{proposition}

Thus, in the setting above, properly discontinuous actions can be attributed to the study of discrete subgroups of $\Isom (X)$.

On the other hand, when a metric on $X$ is not fixed, the situation becomes more complicated.
In fact, an action of discrete group on $X$ such that each orbit is closed and each isotropy is compact may fail to be properly discontinuous. 
The study of properly discontinuous and proper actions in this setting  was pioneered by Calabi-Markus \cite{CalabiMarkus1962}, 
Auslander \cite{Auslander1964}, 
Milnor \cite{Milnor1977} 
Kulkarni \cite{Kulkarni1981}, 
Margulis \cite{Margulis1983},
and Kobayashi \cite{Kobayashi89,Kobayashi92Fuji,Kobayashi93,Kobayashi96,Kobayashi1998deformation,Kobayashi-unlimit}. 
Especially, Kobayashi's works in the 1980s and 1990s systematically developed in non-Riemannian setting.
After their works, 
Lipsman, 
Zimmer, 
Benoist, 
Witte, 
Labourie, 
Baklouti, 
Oh, 
Guichard, 
Kassel, 
Kedim, 
Okuda, 
Elaloui, 
Smilga, 
Morita, 
Tojo, 
Kannaka, 
and more researchers expanded \cite{BakloutiElaloutiKedim2016SelbergWeilKobyashiIJM,BakloutiKhlif2007weak, Benoist96, BenoistLabourie1992,
GueritaudGuichardKasselWienhard2017compactification,
kannaka2023zariski,
Lipsman1995proper, Morita2022Cartan,Okuda13, OhWitte2000new, Shalom2000unitary, Smilga2016, Zimmer1994}.
Nowadays, this research is regarded as a key subject in differential geometry.

Motivated by this line of research, we consider the following foundational questions concerning proper actions without assuming a fixed metric on the space.

\begin{question}\label{question:metricverexistance}
Does there exist a proper action on a space that is not compatible with any metric?
\end{question}

As a negative answer to Question \ref{question:metricverexistance} in the Riemannian setting, Palais gave the following necessary condition for the action $\rho$ to be proper.

\begin{theorem}[Palais~\cite{Palais61}]\label{theorem:palais}
    Let $L$ be a Lie group, $X$ a $C^\infty$-manifold and $\rho$ a $C^\infty$-action.
    Assume that the action $\rho$ is proper.
    Then there exists an $L$-invariant Riemann metric on $X$. 
\end{theorem}

Questions \ref{question:metricvercharacter} and \ref{question:metricverexistance} are the questions in metric geometry.
Next, let us consider similar questions for geometry of uniform spaces as below.

\begin{question}\label{question:uniformvercharacter}
     Fix a uniform structure $\mathcal{U}$ such that the $L$-action $\rho$ on $X$ is compatible with $\mathcal{U}$ (i.e.~ equi continuous). 
     How can the properness of $\rho$ be characterized in terms of the behavior of $L$-orbits?
\end{question}

\begin{question}\label{question:uniformverexistance}
    Does there exist a proper action on a space that is not compatible with any uniform structure?
\end{question}

In the appropriate settings, Yoshino gave the answer to Question \ref{question:uniformvercharacter} as below.

\begin{theorem}[Yoshino \cite{Yoshino2006}]\label{theorem:yoshinointrocharacter}
    Let $(X,\mathcal{U})$ be a locally compact separable Hausdorff uniform space.
    Assume that the $L$-action $\rho$ on $X$ is equi continuous (i.e.~ 
     for any $U\in \mathcal{U}$, there exists $V\in \mathcal{U}$ such that $V_L:=\{(lx,ly)\in X\times X\mid l\in L, (x,y)\in V\}$ is a subset of $U$). 
    Then the following conditions on $\rho$ are equivalent.
    \begin{enumerate}
        \item The $L$-action $\rho$ on $X$ is proper.
        \item For each point $x\in X$, the orbit $L\cdot x$ is closed in $X$ and the isotropy $L_x$ is compact. 
    \end{enumerate}
\end{theorem}

Moreover, Yoshino \cite{Yoshino2006} also gave
a negative answer to Question \ref{question:uniformverexistance} in the appropriate setting as below.

\begin{theorem}[Yoshino \cite{Yoshino2006}]\label{theorem:yoshinointro}
    Let $L$ be a locally compact Hausdorff group, $X$  a paracompact locally compact separable Hausdorff space, and $\rho$ a continuous $L$-action on $X$ such that the quotient space $L\backslash X$ is also paracompact. 
    Assume that the action $\rho$ is proper.
    Then there exists a uniform structure $\mathcal{U}$ on $X$ such that the topology induced by the uniform structure $\mathcal{U}$ coincides with the original topology on $X$ and the action $\rho$ on $(X,\mathcal{U})$ is equi continuous.
\end{theorem}

The purpose of this paper is to study proper actions from the perspective of bornology and coarse geometry. 
Bornology was developed by Nlend \cite{HogbeHenri77} in the context of functional analysis.
We can see historical background of Bornology in \cite{Henri1972}.
On the other hand, large-scale geometry is greatly influenced by Gromov's researches \cite{Gromov1987, Gromov1993} on geometric group theory.
Following that development, coarse geometry was introduced by Roe \cite{Roe1993} to study the index theory of complete Riemannian manifolds.
Later, coarse spaces are defined as a generalization of metric spaces by Roe \cite{Roe2003LectureCoarse}.
Although bornology and coarse geometry originated from different fields, Bunke and Engel \cite{BunkeEngel2020} formulated the concept of  ``bornological coarse spaces'' and developed homotopy theory for them.

The studies of proper actions from the perspective of coarse geometry can be found in \cite{Brodskiy2008,LeitnerVigolo2023}.
It should be noted that Leitner and Vigolo studied the notion of ``coarse $\mathcal{B}$-proper actions'' in \cite{LeitnerVigolo2023} (for the details, see Remark \ref{remark:coarseproper}). 
Based on it, in this paper, we consider bornological proper actions.

In the rest of this section, let $(X,\mathcal{B}_X)$ be a bornological space (Definition \ref{def:bornology}), $(L,\mathcal{B}_L)$ a bornological group (Definition \ref{def:bgroup}) and $\rho$ a bornological  $L$-action (Definition \ref{def:baction}).

\begin{definition}
     In this paper, the $L$-action $\rho$ on $X$ is said to be \emph{bornological proper}  if the set $L_{B, B'}:=\{l\in L\mid lB\cap B'\neq \emptyset \}$ belongs to $\mathcal{B}_L$ for each pair $(B,B')$ of elements of $\mathcal{B}_X$.
\end{definition}

In this paper, we focus on the following questions of bornological proper actions, which are similar to Questions \ref{question:metricvercharacter}, \ref{question:metricverexistance}, \ref{question:uniformvercharacter} and \ref{question:uniformverexistance}.

\begin{question}\label{question:coarsevercharacter}
    Fix a coarse structure  $\mathcal{E}$ on $X$ such that the $L$-action $\rho$ on $X$ is compatible with $\mathcal{E}$ (i.e.~ equi controlled; see Definition \ref{def:equicontrolled}). 
    How can the properness of $\rho$ be characterized in terms of the behavior of $L$-orbits?
\end{question}

\begin{question}\label{question:coarseverexistance}
    Does there exist a proper action on a space that is not compatible with any coarse structure?
\end{question}

For each point $x\in X$, let us write $\iota_x$ for the inclusion map from the orbit $L\cdot x$ to $X$ and $T^x$ for the map $L\rightarrow L\cdot x, ~l\mapsto lx$.

As one of the main results of this paper, we give an answer to Question \ref{question:coarsevercharacter} as below.

\begin{theorem}\label{theorem:bpropercharacter}
    Let $\mathcal{E}$ be a coarse structure on $X$.
    Assume that the equation $\mathcal{B}_X=\mathcal{B}_\mathcal{E}$ holds and the $L$-action $\rho$ on $(X,\mathcal{E})$ is equi controlled. 
    Then the following conditions are equivalent.
    \begin{enumerate}
        \item The action $\rho$ is bornological proper.
        \item For any point $x\in X$, the isotropy $L_x$ belongs to $\mathcal{B}_\mathcal{E}$ and the equation $\iota_x^*(\mathcal{B}_\mathcal{E}) = T^x_*(\mathcal{B}_L)$ holds. 
    \end{enumerate}
\end{theorem}

Here, the family $\mathcal{B}_\mathcal{E}$ is the induced bornology from the coarse structure $\mathcal{E}$ (see Proposition \ref{prop:coarseinduced}).
For each point $x\in X$, the bornology $\iota_x^*(\mathcal{B}_X)$ (resp.~$T^x_*(\mathcal{B}_L)$) on $L\cdot x$ 
is the inverse image bornology of $\mathcal{B}_X$ by the map $\iota_x$ (resp. the image bornology of $\mathcal{B}_L$ by the map $T_x$).
These definitions can be found in Definition \ref{definition:induction}.

As our second main result, we provide a completely negative answer to Question \ref{question:coarseverexistance} as follows.

\begin{theorem}\label{theorem:bproperexistance}
    Assume that the $L$-action $\rho$ on $X$ is bornological proper.
    Then  there exists a coarse structure $\mathcal{E}$ on $X$ such that the equation $\mathcal{B}_X=\mathcal{B}_\mathcal{E}$ holds and the action $\rho$ on $(X, \mathcal{E})$ is equi controlled.
\end{theorem}

This paper is organized as follows.
In Section \ref{section:topproper}, we recall previous researches of proper actions.
In Section \ref{section:bornologycoarse}, let us give the definitions of bornological spaces and coarse spaces.
In Section \ref{section:bproper}, an overview of proper actions is derived in terms of topology and bornology.
Finally in Section \ref{section:mainthm}, we propose the definition of bornological proper actions and prove Theorems \ref{theorem:bpropercharacter} and \ref{theorem:bproperexistance}. 

\section{Preliminaries on topological proper actions}\label{section:topproper}

\subsection{Topological proper actions}
In this section, we recall some terminologies for topological proper actions according to \cite{BakloutiKhlif2007weak, Kobayashi92Fuji, Palais61}.
Throughout Section \ref{section:topproper}, let $L$ be a locally compact Hausdorff group, $X$ a locally compact Hausdorff space and $\rho$ a continuous $L$-action on $X$.

\begin{definition}\label{def:topproper}
    We fix our terminologies as below.
    \begin{enumerate}
        \item The action $\rho$ is said to be \emph{proper}  if the subset $L_{C, C'}:=\{l\in L\mid lC\cap C'\neq \emptyset \}$ of $L$ is compact for each pair $(C,C')$ of compact subsets of $X$.
        \item The action $\rho$ is said to be \emph{weakly proper} if  the subset $L_{x, C}:=\{l\in L\mid lx\in C\}$ of $L$ is compact for each point $x\in X$ and a compact subset $C$ of $X$. 
        \item We say that the action $\rho$ has the \emph{property (CI)} if for each point $x \in X$, the isotropy subgroup $L_x:=\{l\in L\mid lx=x\}$ of $L$ is compact.
    \end{enumerate}
    \end{definition}

    \begin{remark}
        Palais \cite{Palais61} originally defined proper actions in the more general setting of a completely regular space $X$.
        The concept of the property (CI) is introduced by Kobayashi \cite{Kobayashi92Fuji}.
    \end{remark}

One can easily observe that any proper action is weakly proper and any weakly proper action has the property (CI).
The converse claims do not hold in general.

Throughout this paper, a \emph{proper map} refers to a map whose inverse image of any compact set is also compact.
Definition \ref{def:topproper} can be rephrased in terms of proper maps as below.

\begin{proposition}\label{prop:propernessandmap}
    The following holds.
    \begin{enumerate}
        \item The action $\rho$ is proper  if and only if  the map 
        \begin{align*}
        T\colon L\times X\rightarrow X\times X, ~(l,x)\mapsto (x,lx)
        \end{align*}
        is a proper map.
        \item The action $\rho$ is weakly proper if and only if  the map 
        \begin{align*}
        L\rightarrow X, ~l\mapsto lx
        \end{align*}
        is a proper map for each $x\in X$. 
        \item The action $\rho$ has the property (CI) if and only if the quotient $\pi^x \colon L\rightarrow L/L_x$ is a proper map for each $x \in X$.
    \end{enumerate}
\end{proposition}

We also note that the term ``compact'' of Definition \ref{def:topproper} can be replaced by ``relatively compact''.

\begin{proposition}\label{prop:cptrelcpt}
    For the $L$-action $\rho$ on $X$, the following hold.
    \begin{enumerate}
    \item\label{prop:cptrelcpt:item:proper} The action $\rho$ is proper  if and only if  the subset $L_{C, C'}$ of $L$ is relatively compact in $L$ for each pair $(C,C')$ of relatively compact subsets of $X$.
    \item\label{prop:cptrelcpt:item:wproper}  The action $\rho$ is weakly proper if and only if  the subset $L_{x, C}$ of $L$ is relatively compact in $L$ for each point $x\in X$ and a relatively compact subset $C$ of $X$. 
    \item\label{prop:cptrelcpt:item:CI}  The action $\rho$ has the property (CI) if and only if  the isotropy subgroup $L_x$ of $L$ is relatively compact in $L$ for each point $x \in X$.
\end{enumerate}
\end{proposition}

\subsection{Characterizations of weakly proper actions and proper actions }\label{subsection:charatop}
In this section, we recall characterizations of weakly proper actions and proper actions of the continuous $L$-action $\rho$ on the locally compact Hausdorff space $X$ proved by Yoshino \cite{Yoshino2006}.

The following two theorems provide more detailed versions of Theorems \ref{theorem:yoshinointrocharacter} and \ref{theorem:yoshinointro}, without the assumption that the space $X$ is separable.

\begin{theorem}[Yoshino \cite{Yoshino2006}]\label{theorem:yoshinowc}
    The following two conditions on the $L$-action $\rho$ on $X$ are equivalent.
    \begin{enumerate}
        \item The action $\rho$ is weakly proper.
        \item The action $\rho$ has the property (CI), each orbit $L\cdot x$ is closed in $X$ and the map
        \[\widetilde{T}^x\colon L/L_x\rightarrow L\cdot x ~([l]\mapsto lx)\] is a homeomorphism for any point $x\in X$.
    \end{enumerate}
\end{theorem}

\begin{theorem}[Yoshino \cite{Yoshino2006}]\label{theorem:yoshinopw}
    Suppose that $X$ is paracompact and the quotient space $L\backslash X$ is also paracompact.
    Then the following two conditions on the $L$-action $\rho$ on $X$ are equivalent. 
    \begin{enumerate}
        \item The action $\rho$ is proper.
        \item The action $\rho$ is weakly proper and there exists a uniform structure $\mathcal{U}$ on $X$ such that the topology on $X$ induced by the uniform structure $\mathcal{U}$ coincides with the original topology on $X$ and the action $\rho$ is equi continuous (i.e.~ for any $U\in \mathcal{U}$, there exists $V\in \mathcal{U}$ such that $L_*V:=\{(lx,ly)\in X\times X\mid l\in L, (x,y)\in V\}$ is a subset of $U$).
    \end{enumerate}
\end{theorem}

\section{Preliminaries on bornological spaces and coarse spaces}\label{section:bornologycoarse}

In this section, let us recall definitions of bornological spaces, coarse spaces and some properties of them.

\subsection{Bornological spaces}
In this section, we give the definition of bornological spaces (see \cite{HogbeHenri77}) and some basic propositions.

\begin{definition}\label{def:bornology}
    Let $X$ be a set.
    A family $\mathcal{B}\subset \mathcal{P}(X)$ of subsets of $X$ is called a \emph{bornology} on $X$ if it satisfies the following three conditions. 
    \begin{enumerate} 
        \item $\mathcal{B}$ is a covering of $X$ (i.e.~ $\bigcup_{B\in \mathcal{B}}B=X$).
        \item $B_1 ,B_2\in \mathcal{B}$ implies $B_1\cup B_2\in \mathcal{B}$. 
        \item $B\in \mathcal{B}$ and $B'\subset B$ implies $B'\in \mathcal{B}$. 
    \end{enumerate}
    The pair $(X,\mathcal{B})$ is called a \emph{bornological space} and each element $B\in \mathcal{B}$ is referred to as a \emph{bounded set}.
\end{definition}

Next, let us give the notion of  ``bornological base''.
For a set $X$, a family $\mathcal{B}^0$ of subsets of $X$ is said to be a \emph{bornological base} on $X$ if it satisfies the following two conditions:

\begin{enumerate}
    \item[(i)'] $\bigcup_{B\in \mathcal{B}^0}B =X$.
    \item[(ii)'] For each pair $(B_1,B_2)$ of elements of $\mathcal{B}^0$, there exists $B\in \mathcal{B}^0$ such that  $B_1\cup B_2\subset B$. 
\end{enumerate}

For each bornological base $\mathcal{B}^0$, the notation $\langle \mathcal{B}^0\rangle$ is defined as follows:
\begin{align*}    
\langle \mathcal{B}^0\rangle:=\{B\subset X\mid \text{there~ exists~} B^0\in \mathcal{B}^0 \text{~such that~} B\subset B^0 \}
\end{align*}
Then the family $\langle \mathcal{B}^0\rangle$ forms a bornology on $X$. 

\begin{example}\label{example:bornology}
    Let us give four examples of bornologies as below. 
    \begin{enumerate}
        \item Let $X$ be a set. Then the power set $\mathcal{P}(X)$ of $X$ is a bornology on $X$. It is called the \emph{maximal bornology}.
        \item Let $(X,d)$ be a metric space. Then the family $\mathcal{B}_d(X)$ of all metrically bounded subsets is a bornology on $X$. 
        It is called the \emph{metric bornology}. 
        \item Let $X$ be a Hausdorff space. Then the family $\mathcal{B}_\mathrm{cpt}(X)$ of all relatively compact subsets is a bornology on $X$. 
        It is called the \emph{compact bornology}.
    \end{enumerate}
\end{example}

Note that for a metric space $(X,d)$, the space is Haine-Borel if and only if the equation $\mathcal{B}_d(X)=\mathcal{B}_\mathrm{cpt}(X)$ holds.

In the rest of this section, let $(X,\mathcal{B}_X)$ and  $(Y,\mathcal{B}_Y)$ be both bornological spaces.

\begin{definition}\label{def:bmaps}
    Let $f$ be a map from $(X,\mathcal{B}_X)$ to $(Y,\mathcal{B}_Y)$.
    \begin{enumerate}
        \item $f$ is called a \emph{bornological map} if $f(B)\in \mathcal{B}_Y$ for every $B\in \mathcal{B}_X$.
        \item $f$ is called a \emph{proper map} if $f^{-1}(D)\in \mathcal{B}_X$ for every $D\in \mathcal{B}_Y$.
    \end{enumerate}
    
\end{definition}

Let us recall the definitions of direct product of bornologies and induction of a bornology. 

\begin{definition}\label{definition:induction}
    Let $S, T$ be sets, $f\colon S\rightarrow X$ a map, and  $\pi \colon X\rightarrow T$ a surjection.
    \begin{enumerate}
        \item The direct product $\mathcal{B}_X\times \mathcal{B}_Y\subset \mathcal{P}(X\times Y)$ forms a bornological base on $X\times Y$.
        We call $\langle \mathcal{B}_X\times \mathcal{B}_Y \rangle$ the \emph{direct product bornology} on $X\times Y$.
        \item The family $f^*\mathcal{B}_X:=\{f^{-1}(B)\subset S\mid B\in \mathcal{B}_X\}$ forms a bornological base on $S$.
        We call $\langle f^*\mathcal{B}_X \rangle $ the \emph{inverse image bornology} of $\mathcal{B}_X$ by the map $f$. 
        Especially if the map $f$ is inclusion map, the bornology $f^*\mathcal{B}_X =\langle f^*\mathcal{B}_X \rangle $ is called the \emph{subbornology} of $\mathcal{B}_X$.
        \item The family $\pi_* \mathcal{B}_X :=\{ \pi (B)\mid B\in \mathcal{B}_X\}$ on $T$ forms a bornology on $T$. We call it the \emph{image bornology} of $\mathcal{B}_X$ by $\pi$. 
    \end{enumerate}
\end{definition}

\subsection{Coarse spaces and bornological coarse spaces }\label{subsection:coarse}
In this section, we set up our notions for coarse spaces (cf.~\cite{DikranZava2020,Roe2003LectureCoarse}). 

\begin{definition}
    Let $X$ be a set.
    A family $\mathcal{E}\subset \mathcal{P}(X\times X)$ is called a \emph{coarse structure} on $X$ if it satisfies the following five conditions.
    \begin{enumerate}
        \item $\mathrm{diag}(X) := \{(x,x) \mid x \in X\} \in \mathcal{E}$.
        \item $E \in \mathcal{E}$ implies $E^T := \{(x',x) \mid (x,x') \in E\} \in \mathcal{E}$.
        \item $E_1, E_2 \in \mathcal{E}$ implies $E_1 \cup E_2 \in \mathcal{E}$.
        \item $E_1, E_2 \in \mathcal{E}$ implies 
        \begin{multline*}
        E_1 \circ E_2 :=\{ (x,z) \in X \times X \mid \\ \text{there~ exists~} y \in X \text{ such that  } (x,y) \in E_1 \text{ and } (y,z) \in E_2  \}\in \mathcal{E}.
        \end{multline*}
        \item $E \in \mathcal{E}$ implies $E' \in \mathcal{E}$ for each $E' \subset E$.
    \end{enumerate}
    The pair $(X, \mathcal{E})$ is called a \emph{coarse space} and 
each element of $\mathcal{E}$ is referred to as a \emph{controlled set}.
\end{definition}

We call a subset  $\mathcal{E}^0$ of $\mathcal{P}(X\times X)$  a \emph{base of a coarse structure} if it satisfies the following four conditions (cf.~\cite{DikranZava2020}):

\begin{description}
    \item[(i)'] $\diag (X) \subset E'$ for some $E' \in \mathcal{E}^0$.
    \item[(ii)'] For each $E \in \mathcal{E}^0$, there exists $E' \in \mathcal{E}^0$ such that $E^T \subset E'$.
    \item[(iii)'] For each pair $(E_1,E_2)$ of elements of  $\mathcal{E}^0$, 
    there exists $E' \in \mathcal{E}^0$ such that $E_1 \cup E_2 \subset E'$.
    \item[(iv)'] For each pair $(E_1,E_2)$ of elements of   $\mathcal{E}^0$, 
    there exists $E' \in \mathcal{E}^0$ such that $E_1 \circ E_2 \subset E'$.
\end{description}

For a base $\mathcal{E}^0$ of a coarse structure on $X$,  we fix the notation $\langle \mathcal{E}^0 \rangle$ as below:
\[\langle \mathcal{E}^0 \rangle := \{ E\subset X\times X \mid   \text{there exists}\hspace{2mm} E^0\in \mathcal{E}^0  \hspace{2mm}\text{such that}  \hspace{2mm} E\subset E^0\}.\]
Then $\langle \mathcal{E}^0 \rangle$ forms a coarse structure on $X$. 

\begin{example}\label{example:coarse}
    Let $(X,d)$ be a metric space and for each $r\geq 0$, define the set 
    \[
    E_r :=\{(x,y)\in X\times X \mid d(x,y)\leq r \}.
    \]
    Then the collection $\mathcal{E}^0_d :=\{E_r \subset X\times X\mid r\geq 0\}$ is a base of a coarse structure. 
    We call the set $\mathcal{E}_d:=\langle \mathcal{E}^0_d\rangle$
    the \emph{bounded coarse structure}.
\end{example}

Next, we fix some terminologies of maps between coarse spaces.

\begin{definition}
    Let $(X,\mathcal{E}), (Y,\mathcal{F})$ be coarse structures and $f$ a map from $X$ to $Y$.
    \begin{enumerate}
        \item The map $f$ is said to be \emph{controlled} if $(f\times f)(E)\in \mathcal{F}$ for any $E\in \mathcal{E}$. 
        \item The map $f$ is said to be \emph{coarsely surjective} if
        there exists a controlled set $F\in \mathcal{F}$ such that $F[f(X)]=Y$. 
    \end{enumerate}
\end{definition}

Hereafter, let us study bornologies induced by coarse structures.

\begin{proposition}\label{prop:coarseinduced} 
    The family 
    \begin{multline*}
    \mathcal{B}_\mathcal{E}:=\{B\subset X\mid \text{there~exists~a~finite~subset }~ A \subset X, \text{~and~} E\in \mathcal{E}
    \\\text{~such ~that~} B\subset E[A]\}
    \end{multline*}
    of subsets of $X$ is a bornology.
    Each element $B\in \mathcal{B}_\mathcal{E}$ is called a \emph{coarsely bounded set}.
\end{proposition}

The following proposition of the bornology $\mathcal{B}_\mathcal{E}$ holds from the definition above. 

\begin{proposition}\label{prop:coarsebornologyproperty}
    For each $E\in \mathcal{E}$ and $B\in \mathcal{B}_\mathcal{E}$, 
    The set $E[B]$ also belongs to $\mathcal{B}_\mathcal{E}$.
\end{proposition}

The proposition below will be applied in Section
\ref{subsection:propercharacter}.
\begin{proposition}[cf.~\cite{Roe2003LectureCoarse}]\label{prop:connected}
    Assume that the coarse space $(X,\mathcal{E})$ is coarsely connected (i.e.~ for every pair $(x,y)\in X\times X$, the singleton $\{(x,y)\}$ is controlled set).
    Then the following hold.
    \begin{enumerate}
        \item For any $B\in \mathcal{B}_\mathcal{E}$, there exists a point $x\in X$ and a controlled set $E\in \mathcal{E}$ such that the equation $B=E[x]$ holds.
        \item For any $B\in \mathcal{B}_\mathcal{E}$, the direct product $B\times B$ is a member of $\mathcal{E}$.
    \end{enumerate} 
\end{proposition}

\begin{remark}
    Roe \cite{Roe2003LectureCoarse} defined bounded sets in coarse spaces slightly different from our definition in Proposition \ref{prop:coarseinduced}.
    Leitner and Vigolo also considered bornologies induced by coarse structures based on Roe's bounded sets.
    definition is also slightly different from our definition in Proposition \ref{prop:coarseinduced}.
    In the case of coarsely connected spaces, our definitions coincide with theirs.
\end{remark} 

For a fixed bornology $\mathcal{B}$ on a set $X$, we shall define the coarse structure $\mathcal{E}_\mathcal{B}$ as below. 

\begin{proposition}\label{prop:coarseassociatedb}
    For a bornological space $(X,\mathcal{B})$, fix the notation $\mathcal{E}_\mathcal{B}$ as below:
        \[
            \mathcal{E}_\mathcal{B}:=\langle (\mathcal{B}\times \mathcal{B}) \cup \{\diag(X)\}\rangle)
        \]
        Then it is coarsely connected coarse structure on $X$ such that the equation $\mathcal{B}_{\mathcal{E}_\mathcal{B}}=\mathcal{B}$ holds.
         It is the minimum among the coarsely connected coarse structures $\mathcal{E}$ with $\mathcal{B}_\mathcal{E}=\mathcal{B}$.
        We call the coarse structure  $\mathcal{E}_\mathcal{B}$ the \emph{coarsely connected structure associated to $\mathcal{B}$}.    
\end{proposition}

\subsection{Group actions on bornological spaces}\label{subsection:Gbornology}

In this section, we recall the definitions of bornological groups and bornological actions on  bornological spaces.

\begin{definition}[\cite{Pombo2012}]\label{def:bgroup} 
    A group $L$ equipped with a bornology  $\mathcal{B}_L$ is called a \emph{bornological group}  if the multiplication 
    \[
    L\times L\rightarrow L, ~(l_1,l_2)\mapsto l_1l_2
    \]
    and the inversion 
    \[
    L\rightarrow L, ~l\mapsto l^{-1}
    \]
    are both bornological maps. 
    Here, the set $L\times L$ is equipped with the direct product bornology
    (see Definition \ref{definition:induction}).
\end{definition}

\begin{remark}
     Nicas and Rosenthal \cite{NicasRosenthal2012} studied  ``generating family for a compatible coarse structure'' on a group.
     We note that for each bornological group $(L,\mathcal{B}_L)$, the bornology $\mathcal{B}_L$ is a generating family for a compatible coarse structure. 
\end{remark}

Subsequently we define a bornological action of a bornological group.

\begin{definition}\label{def:baction}
    Let $X$ be a set and $(L,\mathcal{B}_L)$ a bornological group.
    An $L$-action $\rho$ on $X$ is called a \emph{bornological action} if the map
    \begin{align*}
        \rho \colon L\times X \rightarrow X,~
        (l,x) \mapsto lx
    \end{align*}
    is bornological. 
    Here, $L\times X$ is equipped with the direct product bornology (see Definition \ref{definition:induction}).
\end{definition}

\begin{example}\label{ex:leftbounded}
    Let $(L,\mathcal{B}_L)$ be a bornological group and $\rho$ the natural left $L$-action on $L$ itself.
    Then the action $\rho$ is a bornological.
\end{example}

\subsection{Group actions on coarse spaces}\label{subsection:Gcoarse}
In this section, we give definitions of  \emph{equi controlled} actions and \emph{coarsely transitive} actions on coarse spaces.
Let $(X,\mathcal{E})$ be a coarse space, $L$ a group and $\rho$  an $L$-action on $X$. 

\begin{definition}[\cite{LeitnerVigolo2023}]\label{def:equicontrolled} 
    The $L$-action $\rho$ on the coarse space $(X,\mathcal{E})$ is called \emph{equi controlled} if for each $E \in \mathcal{E}$, there exists $F \in \mathcal{E}$ such that $E_l:=\{(lx, ly)\mid (x,y)\in E\}\subset F$ for every $l\in L$.
\end{definition}

Note that the above definition is rephrased as $E_L:=\bigcup_{l\in L} E_l$ being a controlled set for each $E\in \mathcal{E}$.

\begin{remark}
    Leitner and Vigolo \cite{LeitnerVigolo2023} considered the term ``equi controlled'' for families of maps indexed by a coarse space.
\end{remark}

\begin{example}\label{ex:leftequi}
    Let $L$ be a group, $(X,d)$ a metric space and $\rho$ an isometric $L$-action on $X$.
    Then the action $\rho$ is equi controlled on $(X,\mathcal{E}_d)$. 
    Here, the family $\mathcal{E}_d$ is the bounded coarse structure on $X$ (see Example \ref{example:coarse}). 
\end{example}

We shall define the coarsely transitive group actions on coarsely connected coarse spaces as below.

\begin{definition}\label{def:coarsetransitive}
    Assume that the coarse space $(X,\mathcal{E})$ is coarsely connected and the $L$-action $\rho$ on $X$ is equi controlled.
    The $L$-action $\rho$ on $(X,\mathcal{E})$ is said to be \emph{coarsely transitive} if it satisfies the following equivalent three conditions. 
    \begin{enumerate}
        \item\label{def:transitive:forall} For any point $x\in X$, the orbit map $T^x$ is coarsely surjective (i.e.~ there exists $E\in \mathcal{E}$ such that $E[L\cdot x]=X$).
        \item\label{def:transitive:exists} There exists a point $x\in X$ such that the orbit map $T^x$ is coarsely surjective.
        \item\label{def:transitive:bounded} There exists a set $B\in \mathcal{B}_\mathcal{E}$ such that $L\cdot B=X$ (the bornology $\mathcal{B}_\mathcal{E}$ is given by Definition \ref{prop:coarseinduced}).
    \end{enumerate}
\end{definition}

We note that each transitive action on coarse spaces is coarsely transitive.

The equivalences claimed in Definition \ref{def:coarsetransitive} follows from the proposition below.

\begin{proposition}\label{prop:nbdinclusion}
    Take $l\in L$ and $E,F\in \mathcal{E}$.
    Assume that the inclusion $E_l\subset F$ holds.
    Then the set $l\cdot E[x]$ is a subset of $ F[lx]$ for any point $x\in X$.
\end{proposition}

\subsection{Coarse structures on bornological groups}

Let $(L,\mathcal{B}_L)$ be a bornological group.
We shall define the coarse structure $\mathcal{E}^R_{\mathcal{B}_L}$ as follows. 
 \begin{proposition}\label{prop:leftinversever}
    We fix two notations as below:
    \begin{align*}
        E^D&:=\{(l,h)\in L\times L\mid l^{-1}h\in D\}~\text{for}~ D\in \mathcal{B}_L, \\
        \mathcal{E}^0&:=\{E^D\mid D\in \mathcal{B}_L\}.
    \end{align*}

    Then the family $\mathcal{E}^0$ forms a base of a coarse structure.
    We write $\mathcal{E}^R_{\mathcal{B}_L}$ for the coarse structure induced by the above $\langle \mathcal{E}^0\rangle$. 
 \end{proposition}

\begin{remark}
    Noting that this coarse structure $\mathcal{E}^R_{\mathcal{B}_L}$ has another base as below.
    Let us fix the following notation for each $D\in \mathcal{B}_L$:
     \begin{align*}
    E^{R}_D :=\{(l,h)\in L\times L\mid \text{there~exists~} g\in L \text{~such ~that~} l,h\in gD\}.
    \end{align*}
    Then the family  $\mathcal{E}^0_{\mathcal{B}_L} :=\{E^{R}_D \mid D\in \mathcal{B}_L\}$
    forms a base of the coarse structure $\mathcal{E}^R_{\mathcal{B}_L}$. 
    The coarse structure $\mathcal{E}^R_{\mathcal{B}_L}=\langle \mathcal{E}^0_{\mathcal{B}_L}\rangle$ is denoted by $\mathcal{E}^\mathrm{left}_{\mathcal{E}_\mathrm{min}, \mathcal{B}_L}$ in Leitner and Vigolo's work (see \cite{LeitnerVigolo2023}).
    Here, $\mathcal{E}_\mathrm{min}:=\langle \diag (X)\rangle$ is the minimal coarse structure on the set $L$.
\end{remark}

\section{Proper actions on bornological spaces}\label{section:bproper}

In this section, we focus on bornological proper actions, which generalize the classical notion of topological proper actions.

\subsection{Bornological proper actions}

In this section, we give the definitions of bornological proper actions, weakly bornological proper actions and the property (BI).
Let $(L,\mathcal{B}_L)$ be a bornological group (see Definition \ref{def:bgroup}), $(X,\mathcal{B}_X)$ a bornological space (see Definition \ref{def:bornology}) and $\rho$ a bornological $L$-action on $X$ (see Definition \ref{def:baction}). 

\begin{definition}\label{def:bproper} 
    We fix three terminologies as below.
    \begin{itemize}
        \item The action $\rho$ is said to be \emph{bornological proper} (abbreviated as B-proper in the sequel) if the set 
        \[L_{B, B'}:=\{l\in L\mid lB\cap B'\neq \emptyset \}\]
        is a bounded set in $L$ for each pair $(B,B')$ of elements of $\mathcal{B}_X$.
        \item The action $\rho$ is said to be \emph{weakly bornological proper} (abbreviated as weakly B-proper in the sequel) if the set 
        \[
        L_{x, B}:=\{l\in L\mid lx\in B\}
        \]
         is a bounded set in $L$ for each point $x \in X$ and each set $B\in \mathcal{B}_X$. 
        \item We say that the action $\rho$ has \emph{the property (BI)} if the isotropy $L_x:=\{l\in L\mid lx=x\}  \in \mathcal{B}_L$ for each $x \in X$.
    \end{itemize}
\end{definition}

\begin{remark}\label{remark:coarseproper}
    The notion of ``coarsely $\mathcal{B}$-proper'' can be found in the work of Leitner and Vigolo (the definitions can be seen in Chapter 4 in  \cite{LeitnerVigolo2023}).
\end{remark}

Here is an example of a B-proper action.

\begin{example}
    The left multiplication defines a B-proper action of $(L,\mathcal{B}_L)$ on itself.
\end{example}

One can easily observe that any B-proper action is weakly B-proper and any weakly B-proper action has the property (BI).
The converse claims do not hold in general.
  
\begin{example}\label{ex:bactions}
    Let us give the two actions below.
        \begin{enumerate}
        \item Consider the group $\Z$ and the metric space $(X:=l^2(\Z)\setminus \{0\}, d)$ (here, the metric $d$ is the restriction of the norm on $l^2(\Z)$). Let $\rho$ be a $\Z$-action on $X$ as below:
            \begin{align*}
                \rho \colon \Z \times X\rightarrow X,~(n, \{a_i\}_{i\in \Z})\mapsto \{a_{i+n}\}_{i\in \Z}.
            \end{align*}
            Assume that the space $X$ (resp.~  the group $\Z$) is equipped with  the metric bornology $\mathcal{B}_d(X)$ (resp.~the compact bornology $\mathcal{B}_\mathrm{cpt}(\Z)$).
            Then one can see that the action is fixed point free.
            Thus, the action $\rho$ satisfies the property (BI).
            However, it is not weakly B-proper.
            In fact, let us write $B_1(0)$ for the set $\{x\in X\mid \|x\|\leq 1 \}$ and $\mathbf{a}\in X$ for 
            \begin{equation*}
              a_i=\left\{ 
              \begin{alignedat}{2}   
                1~ & (i=0) \\   
                0~ & (i\neq 0). \\   
              \end{alignedat} 
              \right.
            \end{equation*}
        Then the set $\Z_{\mathbf{a}, B_1(0)}$ is equal to $\Z$.
        Therefore, the action $\rho$ is not weakly B-proper.
        \item Consider the group $\Z$, the set
        \begin{align*}    
        X:=\{(x,y)\in \R^2\mid xy\neq 0\}.
        \end{align*}
        Let us denote a standard metric on $X$ by $d$.
        A $\Z$-action $\rho$ on $X$ is defined as below:
        \begin{align*}
            \rho \colon \Z\times X\rightarrow X,
            ~(n,(x_1,x_2))\mapsto (2^nx_1,2^{-n}x_2)
        \end{align*}
        Consider that the metric space $X$(resp.~ the group $\Z$) is equipped with the metric bornology $\mathcal{B}_d(X)$(resp.~the compact bornology $\mathcal{B}_\mathrm{cpt}(\Z)$).
        Then one can see that the action $\rho$ is fixed point free and weakly B-proper.
        However, it is not B-proper.
        In fact, let us put the set $S:=S^1\cap X$.
        Then the equation $\Z_{S,S}=\Z$ holds.
        \end{enumerate}
    \end{example}

The B-properness of bornological actions obviously depends on bornology on $X$.
The following examples demonstrate this below.

\begin{example}
    Let us consider the two actions.
    \begin{enumerate}
        \item Let us observe the natural shift $\Z$-action on $\R$.
        Denote the standard metric on $\R$ by $d$ and define another metric $d_{1}$ on $\R$ as below:
        \begin{align*}
            d_1(x,y):=\min \{d(x,y), 1\} ~\text{for each } x,y\in \R.
        \end{align*}
        In the setting where $\R$ is equipped with the bornology $\mathcal{B}_d(\R)$, the natural shift action is B-proper.
        However, in the setting of $\R$ equipped with the bornology $\mathcal{B}_{d_1}(\R)$, the natural shift action is not B-proper.
        Moreover, it is not weakly B-proper.
        In fact, the bornology $\mathcal{B}_{d_1}(\R)$ coincides with the power set $\mathcal{P}(\R)$ of $\R$ and the equation $\Z_{0,\R}=\Z$ holds.

        On the other hand, the topologies induced by these two metrics coincide.
        In other words, the two compact bornoloies induced by metrics $d$ and $d_1$ coinsides.
        Therefore, the above example is a phenomenon specific to bornological properness.
        \item Let us recall the setting of (ii) in  Example \ref{ex:bactions}.
        We observed that the $\Z$-action $\rho$ on $(X,\mathcal{B}_d(X))$ is not B-proper.
        Let us consider replacing the metric bornology $\mathcal{B}_d(X)$ with the compact bornology $\mathcal{B}_\mathrm{cpt}(X)$ on $X$.
        Then one can easily check that the action $\rho$ on $(X,\mathcal{B}_\mathrm{cpt}(X))$ is B-proper.
    \end{enumerate}
\end{example}

Definition \ref{def:bproper} can be rephrased as below.

\begin{proposition}\label{prop:bpropernessandmap}
    The following hold.
    \begin{enumerate}
        \item The action $\rho$ is B-proper  if and only if  the map $T\colon L\times X\rightarrow X\times X$ is a proper map (the term ``proper map'' can be seen in Definition \ref{def:bmaps}).
        \item The action $\rho$ is weakly B-proper if and only if  the map 
        \begin{align*}
        L\rightarrow X, ~ l\mapsto lx
        \end{align*}
        is a proper map for each $x\in X$. 
        \item The action $\rho$ has the property (BI) if and only if the quotient map $\pi^x \colon L\rightarrow L/L_x$ is a proper map for each $x \in X$.
    \end{enumerate}
\end{proposition}

The proof is similar to that of Proposition \ref{prop:propernessandmap}.

\subsection{A relationship between bornological proper actions and  topological proper actions}
In this section, we discuss the relationship between topological proper actions (see Definition \ref{def:topproper})  and bornological proper actions (see Definition \ref{def:bproper}). 
Let $L$ be a locally compact Hausdorff group, $X$ a locally compact Hausdorff space and $\rho$ a continuous $L$-action on $X$. 
Consider that the group $L$ and the space $X$ are both equipped with compact bornologies $\mathcal{B}_\mathrm{cpt}(L)$ and $\mathcal{B}_\mathrm{cpt}(X)$ (see Example \ref{example:bornology}).

The following relationship between bornological proper actions and topological proper actions holds from Proposition \ref{prop:cptrelcpt}.

\begin{theorem}\label{theorem:topbor}
    The following  holds.
    \begin{enumerate}
        \item The continuous $L$-action $\rho$ on $X$ is bornological.
        \item The $L$-action $\rho$ on $X$ is proper if and only if it is B-proper.
        \item The $L$-action $\rho$ on $X$ is weakly proper if and only if it is weakly B-proper.
        \item The $L$-action $\rho$ on $X$ has the property (CI) if and only if it has the property (BI).
    \end{enumerate}
\end{theorem}

\section{Main results}\label{section:mainthm}

In this section, we provide characterizations of B-proper actions and weakly B-proper actions (see Definition \ref{def:bproper}).  

\subsection{A Characterization of weakly B-proper actions }\label{subsection:weakcharacter}
In this section, we characterize the weakly B-proper actions.
Let $(L,\mathcal{B}_L)$ be a bornological group (see Definition \ref{def:bgroup}), $(X,\mathcal{B}_X)$ a bornological space and $\rho$ a bornological $L$-action on $X$ (see Definition \ref{def:baction}).
For a point $x\in X$, let $\iota_x$ be the inclusion from $L\cdot x$ to $X$, $T^x$ a surjection $L\rightarrow L\cdot x, l\mapsto lx$ and recall the notations as below:
 \begin{align*}
      \iota_x^*\mathcal{B}_X &:=\{\iota^{-1}_x(B)\subset L\cdot x \mid B\in \mathcal{B}_X \},\\
      T_*^x\mathcal{B}_L &:=\{T^x(D)\subset L\cdot x\mid D\in \mathcal{B}_L\}.
 \end{align*}

Weakly B-proper actions are characterized as below.

\begin{theorem}\label{theorem:weakeqqBI}
The following conditions on the bornological $L$-action $\rho$ on $X$ are equivalent. 
    \begin{enumerate}
        \item\label{theorem:weqqb:item:weakly} The action $\rho$  is weakly B-proper.
        \item\label{theorem:weqqb:item:BI} The equation $\iota_x^*\mathcal{B}_X = T_*^x\mathcal{B}_L$ holds for any $x\in X$ and  the action $\rho$ has the property (BI).
    \end{enumerate}
\end{theorem}
To prove the above theorem, we provide the following 
two propositions.

\begin{proposition}\label{prop:inclusionqiota}
    The inclusion $\iota_x^*\mathcal{B}_X \supset T_*^x\mathcal{B}_L$ holds for each point $x \in X$.
\end{proposition}

\begin{proposition}\label{prop:transitivebounded}
    Let $x$ be a point in $X$.
    Assume that the bornological $L$-action $\rho$ on $X$ is transitive. 
    Then the following two conditions on $\rho$ are equivalent.
    \begin{enumerate}
        \item\label{prop:transitivebounded:proper} 
        For any set $B\in \mathcal{B}_X$, the set $(T^x)^{-1} (B)$ belongs to $\mathcal{B}_L$.
        \item\label{prop:transitivebounded:cpt} 
        The equation $T^x_*\mathcal{B}_L=\mathcal{B}_X$ holds and the set $(T^x)^{-1} (x)$ belongs to $\mathcal{B}_L$.
    \end{enumerate}
\end{proposition}

Proposition \ref{prop:inclusionqiota} follows directly from the definition.
Let us give a proof of Proposition \ref{prop:transitivebounded} as below.

\begin{proof}[Proof of Proposition \ref{prop:transitivebounded}]
    Let us suppose \eqref{prop:transitivebounded:proper} and show \eqref{prop:transitivebounded:cpt}.
    One can see that the set $(T^x)^{-1}(x)$ belongs to $\mathcal{B}_L$ and the inclusion $T^x_*\mathcal{B}_L\subset \mathcal{B}_X$ holds.
    Hence we show the converse inclusion.
    Fix a set $B\in \mathcal{B}_X$.
    Since the action $\rho$ is transitive, the equation $B=T^x((T^x)^{-1}(B))$ holds.
    Thus, the set $B$ also belongs to $T^x_*\mathcal{B}_L$.
    Therefore the equation $T^x_*\mathcal{B}_L=\mathcal{B}_X$ holds.
    
    Next, we prove that \eqref{prop:transitivebounded:cpt} implies \eqref{prop:transitivebounded:proper}.
    Take a set $B\subset \mathcal{B}_X$.
    Our goal is to show that the set $(T^x)^{-1}(B)$ belongs to $\mathcal{B}_L$.
    By the equation $T^x_*\mathcal{B}_L=\mathcal{B}_X$, there exists a member $D$ of $\mathcal{B}_L$ such that the set $T^x(D)$ coincides with $B$.
    Further, the equation $(T^x)^{-1}(T^x(D))=D\cdot (T^x)^{-1}(x)$ holds.
    Thus, the set $(T^x)^{-1}(B)=D\cdot (T^x)^{-1}(x)$ is a member of $\mathcal{B}_L$.
\end{proof}

We prove Theorem \ref{theorem:weakeqqBI} as below.

\begin{proof}[Proof of Theorem \ref{theorem:weakeqqBI}]
    From Propositions \ref{prop:bpropernessandmap} and \ref{prop:transitivebounded}, 
    it immediately holds that 
    \eqref{theorem:weqqb:item:BI} implies \eqref{theorem:weqqb:item:weakly}.
    Hence we only need to show the converse claim.
     It is clear that the weakly proper action $\rho$ has the property (BI).
     Fix a point $x \in X$.
     Our goal is to show the equation $\iota_x^* \mathcal{B}_X = T_*^x\mathcal{B}_L$.
    The inclusion $\iota_x^*\mathcal{B}_X \supset T_*^x\mathcal{B}_L$ holds by Proposition \ref{prop:inclusionqiota}.
    Hence it is enough to check the converse inclusion.
    Take $B \in \mathcal{B}_X$.
    From Proposition \ref{prop:transitivebounded},
    it suffices to show that the set $(T^x)^{-1}(\iota_x^{-1}(B))$ is a member of $\mathcal{B}_L$.
    Since the action $\rho$ is weakly B-proper and the equation $(T^x)^{-1}(\iota_x^{-1}(B))=L_{x, B}$ follows, the set $(T^x)^{-1}(\iota_x^{-1}(B))$ belongs to $\mathcal{B}_L$.  
    Therefore, the equation $\iota_x^* \mathcal{B}_X = T_*^x\mathcal{B}_L$ holds.
\end{proof}

Let us observe the following action from the perspective of Theorem \ref{theorem:weakeqqBI}.

\begin{observation}
Let us consider a bornological action as below:
\begin{align*}
 \rho_1 \colon \mathbb{Z}\times S^1\rightarrow S^1, ~(n, e^{2\pi ir})\mapsto e^{2\pi i(r+\sqrt{2}n)}.
\end{align*}
Here, consider that the space $S^1$ (resp.~the group $\Z$) is equipped with the compact bornology $\mathcal{B}_\mathrm{cpt}(S^1)$ (resp.~the compact bornology $\mathcal{B}_\mathrm{cpt}(\Z)$).
One can see that the action $\rho_1$ is not weakly B-proper.
The claim can also be checked by Theorem \ref{theorem:weakeqqBI}.
In fact, Put $x=e^0=1$ and recall two bornologies  $\iota_{x}^*  \mathcal{B}_\mathrm{cpt}(S^1)$ and $T_*^{x} \mathcal{B}_\mathrm{cpt}(\mathbb{Z})$ as follows:
\begin{align*}
    \iota_{x}^*  (\mathcal{B}_\mathrm{cpt}(S^1)) &=\{\{e^{2\sqrt{2}\pi i n} \mid n\in N \} \mid N\subset \mathbb{Z} \}, \\ 
    T_*^{x} (\mathcal{B}_\mathrm{cpt}(\mathbb{Z})) &=\{\{e^{2\sqrt{2}\pi i n} \mid n\in N \} \mid N\subset \mathbb{Z}; ~\text{finite} \}.
\end{align*}

Thus, $\iota_{x}^*  (\mathcal{B}_\mathrm{cpt}(S^1))\neq  T_*^{x} (\mathcal{B}_\mathrm{cpt}(\mathbb{Z}))$ holds and $\rho_1$ is not weakly B-proper.
\end{observation}

In the rest of Section \ref{subsection:weakcharacter}, let $L$ be a locally compact Hausdorff group, $X$  a locally compact Hausdorff space, and $\rho$  a continuous $L$-action on $X$.
As Theorem \ref{theorem:yoshinowc}, Yoshino gave the characterization of weakly proper actions from a viewpoint of topology on each orbits.
It is worth emphasizing that Theorem \ref{theorem:yoshinowc} can be readily derived from Theorems \ref{theorem:topbor}, \ref{theorem:weakeqqBI}, Proposition~\ref{prop:transitivebounded}, and the following two propositions.

\begin{proposition}\label{prop:TandtildeT}
    Suppose that the action $\rho$ has the property (CI).
    For each point $x\in X$, the following two conditions on $x$ are equivalent.
    \begin{enumerate}
        \item The map $T^x \colon L\rightarrow L\cdot x, ~l\mapsto lx$ is  a proper map.
        \item The map $\widetilde{T}^x\colon L/L_x \rightarrow L\cdot x, ~lL_x\mapsto lx$ is a homeomorphism.
    \end{enumerate}
\end{proposition}

\begin{proposition}\label{prop:closedeqqborno}
    Let $S$ be a subset of the locally compact Hausdorff space $X$ and $\iota$ the inclusion map from $S$ to $X$. 
    Then the following conditions are equivalent.
    \begin{enumerate}
        \item\label{prop:closedeqqborno:closed}  The subset $S\subset X$ is closed in $X$.
        \item\label{prop:closedeqqborno:bornoequation}  The equation $\iota^*\mathcal{B}_\mathrm{cpt}(X)=\mathcal{B}_\mathrm{cpt}(S)$ holds. 
    \end{enumerate}
\end{proposition}

Proposition \ref{prop:TandtildeT} can be checked easily.
Let us give a proof of Proposition \ref{prop:closedeqqborno} as below.

\begin{proof}[Proof of Proposition \ref{prop:closedeqqborno}]
    Firstly, suppose \eqref{prop:closedeqqborno:closed} and prove \eqref{prop:closedeqqborno:bornoequation}.
    One can see that the inclusion $\iota^*\mathcal{B}_\mathrm{cpt}(X)\supset \mathcal{B}_\mathrm{cpt}(S)$ holds in general.
    Hence we only show that  the inclusion $\iota^*\mathcal{B}_\mathrm{cpt}(X)\subset \mathcal{B}_\mathrm{cpt}(S)$ holds. 
    Take $B\in \mathcal{B}_\mathrm{cpt}(X)$. 
    Our goal is to show that the set $\mathrm{cl}_S \iota^{-1}(B)=\mathrm{cl}_S (B\cap S)$ is compact. 
    From the assumption \eqref{prop:closedeqqborno:closed}, the set $\mathrm{cl}_S (B\cap S)$ is also closed in $X$. 
    Since the set $\mathrm{cl}_X B$ is compact and the inclusion $\mathrm{cl}_S (B\cap S)\subset \mathrm{cl}_X B$ holds, the set $\mathrm{cl}_S (B\cap S)$ is also compact and a member of $\mathcal{B}_\mathrm{cpt}(S)$.
    Therefore, the equation $\iota^*\mathcal{B}_\mathrm{cpt}(X)=\mathcal{B}_\mathrm{cpt}(S)$ holds.
    
    Next, we prove that \eqref{prop:closedeqqborno:bornoequation} implies \eqref{prop:closedeqqborno:closed}.
    Take a point $p\in \mathrm{cl}_X S$. 
    It suffices to show $p$ is an element in $S$.
    Since $X$ is locally compact, there exists a compact neighborhood $K_p\in \mathcal{B}_\mathrm{cpt}(X)$ of $p$.
    By the assumption \eqref{prop:closedeqqborno:bornoequation}, 
    there exists $B_S\in \mathcal{B}_\mathrm{cpt}(S)$ such that $K_p\cap S=B_S.$
    Since the set $K_p \cap S$ is closed in $S$, 
    the equation 
    \[
    K_p\cap S= \mathrm{cl}_S B_S=K_p \cap \mathrm{cl}_S B_S 
    \]
    holds.
    In addition, the set $\mathrm{cl}_S B_S$ is compact by $B_S\in \mathcal{B}_\mathrm{cpt}(S)$. Hence the set $K_p\cap S=K_p\cap \mathrm{cl}_S B_S$ is also compact, especially it is closed in $X$.
    Thus, the equation $K_p\cap S=\mathrm{cl}_X (K_p\cap S)$ holds and we only prove that the point $p$ is an element of $\mathrm{cl}_X (K_p\cap S)$.
    For any neighborhood $W_p$ of $p$, the intersection $W_p\cap K_p$ is also the neighborhood of $p$. Hence the set
    $W_p \cap K_p\cap S$ is not the emptyset
    by $p\in \mathrm{cl}_X S$. 
    Thus, the point $p$ is an element of $\mathrm{cl}_X (K_p\cap S)$.
    Therefore, the set $S$ is closed in $X$.
\end{proof}

\subsection{Characterizations of B-proper actions}\label{subsection:propercharacter}

Let $(L,\mathcal{B}_L)$ be a bornological group (see Definition \ref{def:bgroup}), $(X,\mathcal{B}_X)$ a bornological space and $\rho$ a bornological $L$-action on $X$ (see Definition \ref{def:baction}).

For each $E\subset X\times X$ and $l\in L$, 
let us write $E_l$ for the set $\{(lx,ly)\mid (x,y)\in E\}$ and $E_L$ for the union $\cup_{l\in L} E_l$.
Besides, we shall define the notation $E(L,B)$ for each $B\in \mathcal{B}_X$ and put $\mathcal{E}^0(L,\mathcal{B}_X)$ as below:
\begin{align*}
    E(L,B)&:=(B\times B)_L\cup \diag (X)=\{(lb_1,lb_2)\in X\times X\mid l\in L, b_1,b_2\in B\} \cup \diag(X),\\
    \mathcal{E}^0(L,\mathcal{B}_X)&:=\{E(L,B) \subset X\times X\mid B \in \mathcal{B}_X\}.
\end{align*}

As one of the main results, we give a characterization of the B-properness of the action $\rho$ as below. 

\begin{theorem}\label{theorem:proweakeqq}
    The following three conditions on the bornological $L$-action $\rho$ on $X$ are equivalent.
    \begin{enumerate}
        \item\label{theorem:proweakeqq:item:pro} The action $\rho$ is B-proper. 
        \item\label{theorem:proweakeqq:item:str}  The action $\rho$ is  weakly B-proper and the family $\mathcal{E}^0(L,\mathcal{B}_X):=\{E(L,B) \subset X\times X\mid B \in \mathcal{B}_X\}$ is a base of a coarse structure on $X$.
        \item\label{theorem:proweakeqq:item:unibd}  The action $\rho$ is weakly B-proper and there exists a coarse structure $\mathcal{E}$ on $X$ such that the equation $\mathcal{B}_\mathcal{E}=\mathcal{B}_X$ holds and the action $\rho$ is equi controlled on $(X,\mathcal{E})$.
    \end{enumerate}
    In such the situation above, $\mathcal{E}(L,\mathcal{B}_X):=\langle \mathcal{E}^0(L,\mathcal{B}_X)\rangle$ is called a \emph{coarse structure associated by $\rho$}.
    It is the minimum among the coarsely connected coarse structure $\mathcal{E}$ on $X$ such that the equation $\mathcal{B}_\mathcal{E}=\mathcal{B}_X$ holds and the action $\rho$ is equi controlled.
\end{theorem}

To prove Theorem \ref{theorem:proweakeqq}, we propose two lemmas as below. 

\begin{lemma}\label{lem:elbnbd}
    Fix $B\in \mathcal{B}_X$ and $x\in X$.
    Then the following equation holds:
    \[
    E(L,B)[x]=((L_{x, B})^{-1} \cdot B) \cup \{x\}.
    \]
\end{lemma}

\begin{lemma}\label{lem:elb}
    Let  $B, B_1$ and $B_2$ be subsets of $X$.
    Then the following hold:
    \begin{enumerate}
        \item\label{lem:elb:inv} $L_*(E(L,B))=E(L,B)$, 
        \item\label{lem:elb:diag} $\diag (X)\subset E(L,B)$,
        \item\label{lem:elb:t} $E(L,B)^T=E(L,B)$,
        \item\label{lem:elb:union} $E(L,B_1)\cup E(L,B_2) \subset E(L,B_1\cup B_2)$,
        \item\label{lem:elb:composition} $E(L, B_1)\circ E(L, B_2)\subset E(L, (L_{B_1, B_2}\cdot B_1)\cup B_1 \cup B_2)$.
    \end{enumerate}
\end{lemma}

The former lemma is clear, we prove the later as below.

\begin{proof}[Proof of Lemma \ref{lem:elb}]
    Fix subsets $B, B_1$ and $B_2$ of $X$.
    By the definition of $E(L,B)$, four conditions \eqref{lem:elb:inv}', \eqref{lem:elb:diag}', \eqref{lem:elb:t}' and \eqref{lem:elb:union}' follow. 
    Hence we only need to check the condition \eqref{lem:elb:composition}'.
    The following equation holds by the definition of $E(L,B_i)~(i=1,2)$:
    \begin{multline*}
    E(L,B_1)\circ E(L,B_2)\\
    =\bigcup_{l,h\in L}((B_1\times B_1)_l\circ (B_2\times B_2)_h) \cup  (B_1\times B_1)_L \\\cup (B_2\times B_2)_L
    \cup \mathrm{diag}(X).
    \end{multline*}
    These three sets $(B_1\times B_1)_L, (B_2\times B_2)_L$ and $\mathrm{diag}(X)$ are subset of $E(L, (L_{B_1, B_2}\cdot B_1)\cup B_1 \cup B_2)$ from the condition \eqref{lem:elb:union}'.
    Hence it is enough to check that the inclusion
    \[\bigcup_{l,h\in L}((B_1\times B_1)_l\circ (B_2\times B_2)_h)\subset E(L, (L_{B_1, B_2}\cdot B_1)\cup B_1 \cup B_2)\]
    holds.
    For each pair $(l,h)$ of elements of $L$,  in the case of $lB_1\cap hB_2=\emptyset$, the set $(B_1\times B_1)_l\circ (B_2\times B_2)_h$ is the emptyset, otherwise the following equation holds :
    \begin{align*}\label{lem:elbcomposition:inclon1}
    (B_1\times B_1)_l\circ (B_2\times B_2)_h= lB_1\times hB_2=hh^{-1}lB_1\times hB_2.
    \end{align*}
    
    Further, the condition $lB_1\cap hB_2\neq \emptyset$ is equivalent to  $h^{-1} l\in L_{B_1, B_2}$ for any pair $(l,h)$ of elements of $L$. Hence we obtain the following inclusion for every pair $(l,h)$ of elements of $L$:
    \begin{align*}   
    (B_1\times B_1)_l\circ (B_2\times B_2)_h\subset h(L_{B_1, B_2}\cdot B_1)\times hB_2.
    \end{align*}
    Thus, the following inclusion also holds:
    \begin{align*}
    \bigcup_{l,h\in L} (B_1\times B_1)_l\circ (B_2\times B_2)_h&\subset \bigcup_{h\in L}  h(L_{B_1, B_2}\cdot B_1)\times hB_2\\
    &\subset E(L, (L_{B_1, B_2}\cdot B_1)\cup B_1 \cup B_2). 
    \end{align*}
    Therefore, the inclusion 
    \[
    E(L, B_1)\circ E(L, B_2)\subset E(L, (L_{B_1, B_2}\cdot B_1) \cup B_1\cup B_2)
    \]
    follows.
\end{proof}

Let us prove Theorem \ref{theorem:proweakeqq} as below.

\begin{proof}[Proof of Theorem \ref{theorem:proweakeqq}]
    Firstly, we show that \eqref{theorem:proweakeqq:item:pro} implies \eqref{theorem:proweakeqq:item:str}: 
    Suppose \eqref{theorem:proweakeqq:item:pro} and let us check following four conditions.
    \begin{enumerate}[(i)']	
        \item\label{proof:proweakeqq:item1} There exists $E' \in \mathcal{E}^0(L,\mathcal{B}_X)$ such that $\diag (X) \subset E'$.
        \item\label{proof:proweakeqq:item2} For each $E \in \mathcal{E}^0(L,\mathcal{B}_X)$, there exists $E' \in \mathcal{E}^0(L,\mathcal{B}_X)$ such that $E^T \subset E'$.
        \item\label{proof:proweakeqq:item3} For each pair $(E_1,E_2)$ of elements of  $\mathcal{E}^0(L,\mathcal{B}_X)$, 
        there exists $E' \in \mathcal{E}^0(L,\mathcal{B}_X)$ such that $E_1 \cup E_2 \subset E'$.
        \item For each pair $(E_1,E_2)$ of elements of   $\mathcal{E}^0(L,\mathcal{B}_X)$, 
        there exists $E' \in \mathcal{E}^0(L,\mathcal{B}_X)$ such that $E_1 \circ E_2 \subset E'$. \label{proof:proweakeqq:item4}
    \end{enumerate}
    
    \eqref{proof:proweakeqq:item1}', \eqref{proof:proweakeqq:item2}' and  \eqref{proof:proweakeqq:item3}' follow by Lemma \ref{lem:elb}.
    Hence we only need to show that the family $\mathcal{E}^0(L,\mathcal{B}_X)$ satisfies the condition \eqref{proof:proweakeqq:item4}'. 
    Fix elements $E_1$ and $E_2$ of $\mathcal{E}^0(L,\mathcal{B}_X)$.
    Without loss of generality, we may assume that there exists
    a pair $(B_1,B_2)$ of elements of $\mathcal{B}_X$ such that $E_i=E(L,B_i)$ (for each $i=1,2$).
    The following inclusion holds by Lemma \ref{lem:elb}:
    \begin{align*}
        E(L, B_1)\circ E(L, B_2)\subset E(L, (L_{B_1, B_2}\cdot B_1) \cup B_1 \cup B_2).
    \end{align*}
    The set  $L_{B_1, B_2}\cdot B_1$ is a member of the bornology $\mathcal{B}_X$ since the action $\rho$ is B-proper.  
   Thus, the set $(L_{B_1, B_2}\cdot B_1) \cup B_1 \cup B_2$ also belongs to $\mathcal{B}_X$.
    Therefore, the family $\mathcal{E}^0(L,\mathcal{B}_X)$ is a base of a coarse structure. 
   
    Secondly, we show that \eqref{theorem:proweakeqq:item:str} leads to \eqref{theorem:proweakeqq:item:unibd}. 
    Our goal is to show that the equation $\mathcal{B}_{\mathcal{E}(L,\mathcal{B}_X)}=\mathcal{B}_X$ holds and the action $\rho$ on $(X, \mathcal{E}(L,\mathcal{B}_X))$ is equi controlled. 
     From Lemma \ref{lem:elb}, one can see that the action $\rho$ is equi controlled.
    Hence let us prove the equation $\mathcal{B}_{\mathcal{E}(L,\mathcal{B}_X)}=\mathcal{B}_X$. 
    We check the inclusion $\mathcal{B}_{\mathcal{E}(L,\mathcal{B}_X)}\supset \mathcal{B}_X$ earlier.
    Take $B\in \mathcal{B}_X$ and a point $x\in B$. 
    From Lemma \ref{lem:elbnbd},  the set $E(L,B)[x]$ is equal to $((L_{x, B})^{-1} \cdot B) \cup \{x\}$.
    Since the unit element $e\in L$ also belongs to the set $L_{x,B}$,
    the inclusion $B\subset E(L,B)[x]$ follows.
    Thus, the set $B$ is a member of $\mathcal{B}_{\mathcal{E}(L,\mathcal{B}_X)}$.
    Therefore, the inclusion $\mathcal{B}_{\mathcal{E}(L,\mathcal{B}_X)}\supset \mathcal{B}_X$ holds.
    Next, we check the converse inclusion.
    Take finite points $x_1, \cdots ,x_n \in X$ and a controlled set $E\in \mathcal{E}$.
    Without loss of generality, we may assume that there exists an element $B\in \mathcal{B}_X$ such that $E=E(L,B)$.
    Our goal is to show that  the set $\bigcup_{i=1}^n E(L,B)[x_i]$ belongs to $\mathcal{B}_X$.
    By Lemma \ref{lem:elbnbd}, the following inclusion holds:
        \[
        \bigcup_{i=1}^n E(L,B)[x_i] \subset \bigcup_{i=1}^n (((L_{x_i, B})^{-1} \cdot B) \cup \{x_i\}).
        \]
    The set $\bigcup_{i=1}^n (((L_{x_i, B})^{-1} \cdot B) \cup \{x_i\})$ is a member of the bornology $\mathcal{B}_X$ since the action $\rho$ is weakly B-proper.
    Thus, the set $\bigcup_{i=1}^n E(L,B)[x_i]$ belongs to $\mathcal{B}_X$
    and the equation $\mathcal{B}_{\mathcal{E}(L,\mathcal{B}_X)}=\mathcal{B}_X$ follows.
    Therefore, we obtain that  \eqref{theorem:proweakeqq:item:str} leads to \eqref{theorem:proweakeqq:item:unibd}. 
    
    Thirdly, we prove that  \eqref{theorem:proweakeqq:item:unibd} implies \eqref{theorem:proweakeqq:item:pro}. 
    Let us take $D_1, D_2\in \mathcal{B}_X$. 
    Our goal is to show that the set $L_{D_1, D_2}$ is a member of the bornology $\mathcal{B}_L$.
    By the equation $\mathcal{B}_X=\mathcal{B}_\mathcal{E}$, there exists finite points $x_1, \cdots ,x_n \in X$ and $E\in \mathcal{E}$ such that $D_1\subset \bigcup_{i=1}^n E[x_i]$.
    Hence we obtain that 
    \begin{align}\label{align:twothreeinclusion1}
    L_{D_1, D_2} &\subset \{l\in L\mid l\cdot  (\bigcup_{i=1}^n E[x_i])\cap D_2 \neq \emptyset\}. 
    \end{align}
    Since the action $\rho$ is equi controlled on $(X, \mathcal{E})$, there exists $F\in \mathcal{E}$ such that  $ E_l\subset F$ for any $l\in L$.
    From Proposition \ref{prop:nbdinclusion}, the following inclusion holds for each $i=1, \cdots, n$:
    \begin{align}\label{align:twothreeinclusion2}
        \{l\in L\mid l\cdot  (E[x_i])\cap D_2 \neq \emptyset\} \subset \{l\in L\mid  (F[l x_i]\cap D_2) \neq \emptyset\}.
    \end{align}
    By \eqref{align:twothreeinclusion1} and \eqref{align:twothreeinclusion2}, we obtain the following inclusion: 
    \[
    L_{D_1, D_2} \subset \bigcup_{i=1}^n \{l\in L\mid l x_i \cap F^T[D_2] \neq \emptyset\} = \bigcup_{i=1}^n (L_{x_i, F^T[D_2]}).  
    \]
    The set $F^T[D_2]$ is a member of $\mathcal{B}_\mathcal{E}$ from Proposition \ref{prop:coarsebornologyproperty}.
    Since the equation $\mathcal{B}_\mathcal{E}=\mathcal{B}_X$ holds and the action $\rho$ is weakly B-proper,
     the subset $L_{x_i, F^T[D_2]}\subset L$ belongs to $\mathcal{B}_L$ for each $i=1,\cdots ,n$.
    Thus, the set $L_{D_1, D_2}$ also a member of $\mathcal{B}_L$. 
    Therefore, the implication  \eqref{theorem:proweakeqq:item:unibd} $\Rightarrow$ \eqref{theorem:proweakeqq:item:pro} holds.
    
    Finally, assume that the $L$-action $\rho$ on $(X,\mathcal{B}_X)$ is B-proper and take a coarsely connected coarse structure $\mathcal{E}$ on $X$ such that the equation $\mathcal{B}_\mathcal{E}=\mathcal{B}_X$ holds and the action $\rho$ is equi controlled.
    Our goal is to show that the inclusion $\mathcal{E}(L,\mathcal{B}_X)\subset \mathcal{E}$ holds. 
    Fix $E\in \mathcal{E}(L,\mathcal{B}_X)$.
    Without loss of generality, we can assume that there exists $B\in \mathcal{B}_X$ such that $E=E(L,B)$ holds.
    The set $B\times B$ is controlled set in $(X,\mathcal{E})$ from the coarsely connectedness of $\mathcal{E}$.
    Hence we obtain that $E(L,B)$ is a member of $\mathcal{E}$ since the action $\rho$ is equi controlled on $(X,\mathcal{E})$.
    Thus, the inclusion $\mathcal{E}(L,\mathcal{B}_X)\subset \mathcal{E}$ follows.
    Therefore, the associated coarse structure $\mathcal{E}(L,\mathcal{B}_X)$ with $\rho$ is the minimum among the coarsely connected coarse structure $\mathcal{E}$ on $X$ such that the equation $\mathcal{B}_\mathcal{E}=\mathcal{B}_X$ holds.
\end{proof}

\begin{remark}
    In the case where $X$ is a Hausdorff space  equipped with the compact bornology (see Example \ref{example:bornology}), the implication \eqref{theorem:proweakeqq:item:pro} $\Rightarrow$ \eqref{theorem:proweakeqq:item:str} in Theorem \ref{theorem:proweakeqq} can be found in Nicas and Rosenthal \cite{NicasRosenthal2012}.
\end{remark}

\subsection{Coarse structures associated with B-proper actions}\label{subsection:cstrbproper}

In this section, we study some examples of coarse structures associated with B-proper actions (see Theorem \ref{theorem:proweakeqq}).

\begin{proposition}\label{prop:poweraction}
    Let $L$ be a group, $(X,\mathcal{B}_X)$ a bornological space and $\rho$ a $L$-action on $X$.
    Consider that the group $L$ is equipped with the maximal bornology $\mathcal{P}(L)$ and assume that the action $\rho$ is bornological.
    Then the action $\rho$ is B-proper.
    Furthermore, the equation $\mathcal{E}(L,\mathcal{B}_X) =\mathcal{E}_{\mathcal{B}_X}$ holds (the notation $\mathcal{E}_{\mathcal{B}_X}$ is defined in \ref{prop:coarseassociatedb}). 
\end{proposition}

The proof is straightforward.
As a corollary, the following holds.

\begin{corollary}
    Let $L:=\mathrm{O}(n)$ and  $X:=\mathbb{R}^n$.
    We write $\rho$ for the natural $L$-action on $X$ and $d$ for the standard metric on $\mathbb{R}^n$.
    Then the action $\rho$ is B-proper and the equation $\mathcal{E}(L, \mathcal{B}_d(X)) =\mathcal{E}_{\mathcal{B}_d(X)}$ holds.
     It is worth emphasizing that the coarse structure $\mathcal{E}(L, \mathcal{B}_d(X))$ is strictly included in the bounded coarse structure $\mathcal{E}_d$.
\end{corollary}
   
Next, let us study coarsely transitive (see Definition \ref{def:coarsetransitive}) B-proper actions as below.

\begin{theorem}\label{theorem:coarseandinducecoarse}
        Let $(L,\mathcal{B}_L)$ be a bornological group and $(X,\mathcal{E})$ a coarsely connected coarse space and $\rho$ a bornological  B-proper $L$-action on $(X,\mathcal{B}_\mathcal{E})$ ($\mathcal{B}_\mathcal{E}$ is defined in Proposition \ref{prop:coarseinduced}).
        Then the followings hold.
        \begin{enumerate}
        \item\label{theorem:coarseandinducecoarse:item:general} The inclusion $\mathcal{E}(L, \mathcal{B}_\mathcal{E}) \subset \mathcal{E}$ holds.
            \item\label{theorem:coarseandinducecoarse:item:transitive} Suppose that the action $\rho$ is coarsely transitive and equi controlled on $(X,\mathcal{E})$. Then the converse inclusion holds.
        \end{enumerate}
\end{theorem}

\begin{proof}[Proof of Theorem \ref{theorem:coarseandinducecoarse}]
    The statement \eqref{theorem:coarseandinducecoarse:item:general} immediately holds from the equi controlledness of the action $\rho$.
    Hence we only show \eqref{theorem:coarseandinducecoarse:item:transitive}.
    Fix $E\in \mathcal{E}$. 
    Without loss of generality, we may assume that the controlled set $E$ include the diagonal set $\diag(X)$. 
    There exists a point $x\in X$ and $F\in \mathcal{E}$ such that $F[L\cdot x]=X$ from the coarse transitivity of the action $\rho$.
    Since the action $\rho$ is equi controlled,  the sets $E_L:=\{(ls,lt)\mid l\in L, (s,t)\in E\}$ and $F_L:=\{(lp,lq)\mid l\in L, (p,q)\in F\}$ belong to $\mathcal{E}$.
    Our goal is to show that the controlled set $E$ is a subset of $E(L,(E_L\circ F_L)[x])$.
    Fix a pair $(y,z)\in E$.
    Then there exists an element $l\in L$ such that $z$ belongs to $F[lx]$ from the coarse transitivity of $\rho$.
    Hence the pair $(l^{-1}z,x)$ is an element of $F_L$.
    Thus, the element $(l^{-1}y,x)\in X\times X$ belongs to $E_L\circ F_L$.
    As a result, both of the points $l^{-1}y$ and $l^{-1}z$ are elements of $(E_L\circ F_L)[x]$.
    Hence the pair $(y,z)=(ll^{-1}y,ll^{-1}z)$ belongs to  $E(L,(E_L\circ F_L)[x])$.
    Thus, the set $\mathcal{E}(L,\mathcal{B}_\mathcal{E})$ is a subset of $\mathcal{E}$.
    Therefore, the equation $\mathcal{E}(L,\mathcal{B}_\mathcal{E}) =\mathcal{E}$ holds.
\end{proof}

\begin{corollary}\label{cor:leftcoarse}
    Let $(L,\mathcal{B}_L)$ be a bornological group (see Definition \ref{def:bgroup}) and $\rho$  the left $L$-action on itself.
    Then the  equation
    $\mathcal{E}(L, \mathcal{B}_L)= \mathcal{E}^{R}_{\mathcal{B}_L}$
    holds (the coarse structure $\mathcal{E}^R_{\mathcal{B}_L}$ is given in Proposition \ref{prop:leftinversever}).
\end{corollary}

\begin{remark}
Corollary \ref{cor:leftcoarse} can be found in \cite{NicasRosenthal2012}.
\end{remark}

Moreover, in the setting of proper isometric transitive actions on a metric space,  Theorem \ref{theorem:coarseandinducecoarse} can be stated as follows.

\begin{corollary}\label{cor:transitiveisom}
    Let $(L,\mathcal{B}_L)$ be a bornological group, $(X,d)$ a metric space.
    Consider that the metric bornology $\mathcal{B}_d$ (see Example \ref{example:bornology}) on $X$.
    We fix $\rho$ an isometric B-proper transitive $L$-action on $X$.
    Then the equation $ \mathcal{E}(L, \mathcal{B}_d)=\mathcal{E}_d$ holds (the definition of the bounded coarse structure $\mathcal{E}_d$ can be seen in Example \ref{example:coarse}).
\end{corollary}

Here are two examples of Corollaries \ref{cor:leftcoarse} and \ref{cor:transitiveisom}.

\begin{example}
    Let $(V,\|\cdot \|)$ be a normed vector space.
    We put $d$ for the metric induced by norm $\|\cdot \|$ and $\rho$ for the natural shift $V$-action on itself.
    Consider the bounded coarse structure $\mathcal{E}_d$ on $V$ (see Example \ref{example:coarse}) and the metric bornology $\mathcal{B}_d$ on $V$ (see Example \ref{example:bornology}).
    Then the equation $\mathcal{E}({V, \mathcal{B}_d})=\mathcal{E}^R_{\mathcal{B}_d}=\mathcal{E}_d$ follows from  Corollaries \ref{cor:leftcoarse} and \ref{cor:transitiveisom}. 
\end{example}

\begin{example}
    Take $n\in \mathbb{Z}_{>0}$ and represent  the indefinite orthogonal group $\mathrm{SO}(2,2n)$ as $G$.
    Denote a closed subgroup $\mathrm{SO}(1,2n)$ of $G$ by $H$ and $\mathrm{U}(1,n)$ by $L$.
    We write $X$ for the homogeneous space $G/H$ and $\rho$ for the natural $L$-action on $X$.
    It is well-known that the $L$-action $\rho$ on $X:=G/H$ is proper in the sense of topology (cf.~\cite{Kobayashi89}).
    Consider the compact bornology $\mathcal{B}_\mathrm{cpt}(X)$ on $X$ and $\mathcal{B}_\mathrm{cpt}(L)$ on $L$ (see Example \ref{example:bornology}).
    Then the $L$-action $\rho$ on $X$ is B-proper from Theorem \ref{theorem:topbor}. 
    \begin{proposition}\label{prop:indefiniteunital}
    There exists a Riemannian metric $d$ on $X$ such that the metric bornology $\mathcal{B}_d$ coincides with $\mathcal{B}_\mathrm{cpt}(X)$ and 
    the equation $\mathcal{E}(L,\mathcal{B}_\mathrm{cpt}(X)) = \mathcal{E}_d$ holds.
    \end{proposition}
    By Palais's Theorem~\ref{theorem:palais}, there exists an $L$-invariant Riemannian metric on $X$.
    In particular, in this setting, we can explicitly construct such a metric $d$ via the homogeneous space representation $X \cong \mathrm{U}(1,n)/\mathrm{U}(n)$.
    Moreover, the coarse structure $\mathcal{E}(L,\mathcal{B})$ associated with $\rho$ coincides with the bounded coarse structure $\mathcal{E}_d$ induced by $d$.
    \begin{proof}[Proof of Proposition \ref{prop:indefiniteunital}]
    One can see that the $L$-action $\rho$ on $X$ is transitive with the compact isotropy $\mathrm{U}(n)$. 
    Hence the topological homeomorphism $X \cong \mathrm{U}(1,n)/\mathrm{U}(n)$ follows.
    Further, the homogeneous space $\mathrm{U}(1,n)/ \mathrm{U}(n)$ has the $L$-invariant Riemann Haine-Borel metric $d$.
    Thus, we can assume that $X$ has $L$-invariant Riemann Haine-Borel metric $d$.
    Since the $L$-action $\rho$ is transitive and equi controlled on $(X,\mathcal{E}_d)$, Theorem \ref{theorem:coarseandinducecoarse} lead to the equation $\mathcal{E}(L,\mathcal{B}_\mathrm{cpt}(X)) = \mathcal{E}_d$.
    \end{proof}
\end{example}

\section*{Acknowledgments}
The author would like to give heartfelt thanks to Takayuki Okuda for his encouragement to write this paper. 
We are also indebted to 
Toshiyuki Kobayashi,
Hiroshi Tamaru, 
Taro Yoshino,
Takamitsu Yamauchi, 
Shinichi Oguni, 
Tomohiro Fukaya, 
Masato Mimura,  
Akira Kubo,  
Hiroki Nakajima, 
Daisuke Kazukawa, 
Kazuki Kannaka, 
Takumi Matsuka, 
Kento Ogawa, 
Akifumi Nakada, 
Shunsuke Miyauchi, 
and Hikozo Kobayashi
for many helpful comments.
This work is supported by JST SPRING, Grant Number JPMJSP2132.

\providecommand{\bysame}{\leavevmode\hbox to3em{\hrulefill}\thinspace}
\providecommand{\MR}{\relax\ifhmode\unskip\space\fi MR }
\providecommand{\MRhref}[2]{%
  \href{http://www.ams.org/mathscinet-getitem?mr=#1}{#2}
}
\providecommand{\href}[2]{#2}

\end{document}